\documentclass[runningheads]{llncs}
\usepackage{amsmath}
\usepackage{amssymb}
\usepackage{graphicx}
\usepackage{proof}
\usepackage{url}
\DeclareMathOperator{\BL}{\mathbf{BCL}}
\DeclareMathOperator{\B}{\mathbf{B}}
\DeclareMathOperator{\LO}{\mathbf{L}}
\DeclareMathOperator{\C}{\mathbf{C}}
\DeclareMathOperator{\D}{\mathcal{D}}
\DeclareMathOperator{\RAA}{\textup{RAA}}

\DeclareMathOperator{\tonk}{\mathbf{tonk}}
\DeclareMathOperator{\PSet}{\mathcal{P}}
\newcommand{\yoriyuki}[1]{}
\allowdisplaybreaks

\begin{document}
\title{Notion of validity for the bilateral classical logic}
\author{Ukyo Suzuki\inst{1} \and
Yoriyuki Yamagata\inst{2}\orcidID{0000-0003-2096-677X}}
\authorrunning{U. Suzuki and Y. Yamagata}
\institute{
Pico lab\\
\and
National Insistute of Advanced Industrial Science and Technology (AIST), Ikeda, Japan\\
\email{yoriyuki.yamagata@aist.go.jp}
}
\maketitle              % typeset the header of the contribution
\begin{abstract}
  This paper explores proof-theoretic semantics, a formal approach to inferential semantics. It derives sentence meaning from formalized proofs, building upon Gentzen and Prawitz's work. The study addresses challenges in understanding how proofs contribute to sentence meaning.
  
  In this setting, classical logic poses "Dummett's challenge" due to its mismatch with the proof-theoretic framework designed for intuitionist logic. 
  For example, in Rumfitt's bilateral classical logic (BCL), the justification of coordination rules, notably RAA, is contentious.
  
  This paper employs the notion of validity, introduced by Prawitz, to provide a comprehensive justification for BCL, defining valid arguments and demonstrating its soundness. It resolves the circularity associated with RAA using fixed-point construction. Notably, this approach relies on impredicative comprehension but without on the excluded middle principle, suggesting that bivalence may not be essential for justifying classical logic.
\end{abstract}

\section{Introduction}

Proof-theoretic semantics is a formal approach to inferential semantics, where the meaning of a sentence is determined by how it is verified through formalized proofs. This approach was developed based on natural deduction proof systems by Gentzen and Prawits.  It seeks to address challenges in understanding how proofs contribute to the meaning of sentences, under the maxim of "meaning is use".

Gentzen's insight plays a crucial role in proof-theoretic semantics. He associates the meaning of sentences primarily with introduction (I) rules in the proof system, while elimination (E) rules are seen as consequences of these introductions. This idea introduces the concept of justification of proofs and helps in avoiding paradoxical operators like $\tonk$.

Two major approaches have been developed to formalize Gentzen's suggestion: the local and global approaches. In the local approach, the fundamental units of justification are rules, where I-rules are self-justifying, and E-rules must be shown to be in harmony with I-rules. Harmony is typically defined using various criteria, including the leveling of local peaks. In the global approach, the units of justification are arguments, which are proof figures with arbitrary steps. Canonical proofs ending with meaning-conferring rules are considered self-justifying, and non-canonical proofs can be transformed into canonical proofs through a normalization procedure. Valid arguments define the semantics, expressing the proof-theoretical justifiability of proofs.

In the context of classical logic, there is a challenge known as "Dummett's challenge," which involves providing a satisfactory proof-theoretic semantics for classical logic. This challenge arises because classical logic does not readily fit the proof-theoretic framework developed for intuitionist logic. Various approaches have been explored, Rumfitt's bilateral formulation of classical logic (BCL). 

BCL introduces positive and negative contexts and coordination rules using force operators. It defines operational and coordination rules for various logical constants, addressing the challenge of providing a proof-theoretic semantics for classical logic. However, the justification of coordination rules, such as RAA, has been a point of contention. Rumfitt's local approach does not provide a satisfactory proof-theoretic justification for these rules, and Rumfitt resorted to the notion of truth for this purpose~\cite{Rumfitt2008}.

This paper offers a justification for BCL using the global approach. It defines a set of valid arguments in BCL and demonstrates that BCL is sound under the such semantics.  It addresses the circular nature of RAA by introducing a technique based on fixed-point construction.

It is worth noting that our approach utilizes fixed-point construction, a tool often employed in the study of classical logic in computer science. While this technique may introduce impredicative comprehension, it is important to underscore that our approach does not rely on the excluded middle principle. As a result, our framework suggests that the principle of bivalence may not be an indispensable requirement for justifying classical logic. Furthermore, our semantics exemplifies how a meta-theory can provide a justification for an object theory without relying on a principle intrinsic to the object theory itself.

This paper is organized as follows.
Section~\ref{sec:related_works} discusses the related works.
Section~\ref{sec:bcl} introduces BCL, a formal system focused in this paper.
Section~\ref{sec:validity} defines the notion of validity, in the style of Prawits, to BCL.
Section~\ref{sec:justification} justifies all rules of BCL, including the law of contradiction and RAA.

\section{Related works}\label{sec:related_works}

Proof-theoretic semantics, including Prawits's notion of validity~\cite{Prawitz1971} is renowned for its applicability to intuitionistic logic while facing difficulties in validating classical laws.

Consequently, Dummett posits that proof-theoretic justifiability favors intuitionist logic over classical logic, a challenge aptly labeled as "Dummett's challenge" by K\"urbis~\cite{Kurbis2016-KRBSCO}.

In pursuit of resolving Dummett's challenge, several approaches have emerged. For instance, Rumfitt~\cite{Rumfitt2000} introduces force operators for assertion and denial, accompanied by coordination rules, the law of contradiction, and $\RAA$ (Reductio ad absurdum). Rumfitt characterizes coordination rules as structural rules, offering limited proof-theoretic justification by demonstrating sufficiency for the law of contradiction in atomic sentences. In contrast, this paper presents Prawitz's style validity~\cite{Prawitz1971} for BCL, justifying all rules in the proof-theoretic means, including coordination rules.

Our notion of validity relies on fixed-point construction, validated by Knaster-Tarski's theorem~\cite{Tarski1955}. Fixed-point construction has been instrumental in proving normalization results for formal systems influenced by classical logic~\cite{Barbanera1994,Yamagata2001,Yamagata2004a}. BCL can be perceived as an extension of these systems to encompass disjunction. Notably, our validity concept implies weak normalizability, establishing that BCL is weakly normalizable.

It is noteworthy that the proof of Knaster-Tarski's fixed point theorem necessitates impredicative comprehension, while it does not rely on the principle of excluded middle~\cite{Curi2015-ji}. This observation raises the intriguing possibility that while we do not present a constructive account of classical logic, we demonstrate that the principle of bivalence may not be an indispensable requirement for justifying classical logic.

\section{Mathematical Preliminary}
\begin{theorem}[Knaster-Tarkski's fixed point theorem]\label{thm:KT}
  Let $S$ be the set and $F \colon \PSet(S) \to \PSet(S) $.
  Further, we assume that $F$ is monotone, that is \[
    X \subseteq Y \implies F(X) \subseteq F(Y)
  \]
  for any $X, Y \subseteq S$.
  Then, $F$ has a least fixed point $X_0 = F(X_0)$
  and for any $Z, Z = F(Z)$, $X_0 \subseteq Z$.
\end{theorem}

\begin{proof}
  Let \[
  X_0 = \bigcap_{X \subseteq S} \{X \mid F(X) \subseteq X\}
  \]
  Let $\Phi = \{X \mid F(X) \subseteq X\}$.
  Assume that $F(X) \subseteq X$.
  Then, $X \in \Phi$ therefore $X_0 \subseteq X$.
  By monotonicity of $F$, $F(X_0) \subseteq F(X) \subseteq X$.
  \begin{equation}\label{eq:subseteq}
    F(X_0) \subseteq \bigcap \Phi = X_0
  \end{equation}
  Therefore, $F(F(X_0)) \subseteq F(X_0)$ by monotonicity of $F$.
  This implies $F(X_0) \in \Phi$.
  Therefore, $X_0 \subseteq F(X_0)$.
  Combining \eqref{eq:subseteq}, $X_0 = F(X_0)$.
  Because any fixed point $Y$ of $F$ is included in $\Phi$, $X_0$ is the least fixed point.
\end{proof}

The proof presented here uses impredicative comprehension, but does not use the principle of excluded middle.
In fact, Knaster-Tarski's theorem serves as a notable illustration where impredicative comprehension is necessary, yet it doesn't depend on the exclusion of the middle principle~\cite{Curi2015-ji}.

\section{Bilateral language and its formal systems}\label{sec:bcl}

In this section, we introduce the \emph{bilateral language} (abbreviated as BL) and the general notion of \emph{formal systems} for BL.

\emph{Bilateral language} is a language based on the idea that in classical logic, statements can have two linguistic forces, not only affirmation but also, denial.
The system is most famously proposed by Rumfitt \cite{Rumfitt2000} in the context of philosophy, but similar ideas appear in the other literature \cite{Parigot2000,Stewart2000} in the context of computer science.
In this paper, we call each representation of a linguistic act a \emph{statement}, and its content a \emph{proposition}.
Formally, we define
\begin{definition}[Proposition, Statement]
  Atomic propositions are denoted by symbols $p, q, p_1, \ldots$.
  Propositions $A, B, A_1, \ldots$ are defined by
  \begin{equation}
    A := p \mid A \wedge A \mid A \vee A \mid \neg A \mid A \to A.
  \end{equation}
  Statements $\alpha, \beta, \alpha_1, \ldots$ are defined by
  \begin{equation}
    \alpha := +A \mid -A.
  \end{equation}
  In addition, a special symbol $\bot$ appears in derivations.
  $\bot$ should be understood as a punctuation symbol, not a statement.
\end{definition}

\begin{definition}
  $+p$ and $-p$ for an atomic proposition $a$ are called \emph{atomic statements}.
  For a statement $\alpha$, its \emph{conjugate} $\alpha^*$ is defined as
  \begin{align}
    (+A)^* \equiv -A & & (-A)^* \equiv +A
  \end{align}
\end{definition}

We consider the following set BCL  of rules, which is divided into \emph{logical rules} $\LO$ and \emph{coordination rules} $\C$, as a ``natural" set of rules and show that all ``valid" rules can be reduced to them in some sense.
BCL is called \emph{bilateral classical logic}.
\begin{definition}[Logical rules]
  {\small
  \begin{align*}
    \begin{gathered}
      \infer[+\wedge I]{+A \wedge B}{
      +A &+ B
      }
    \end{gathered}
     & &
     \begin{gathered}
       \infer[+\wedge E1]{+A}{
         +A \wedge B
       }
     \end{gathered}
     & &
     \begin{gathered}
       \infer[+\wedge E2]{+B}{
         +A \wedge B
       }
     \end{gathered}
    \\
    \begin{gathered}
      \infer[-\wedge I1]{-A \wedge B}{
        -A
      }
    \end{gathered}
     & &
    \begin{gathered}
      \infer[-\wedge I2]{-A \wedge B}{
        -B
      }
    \end{gathered}
     & &
    \begin{gathered}
      \infer[-\wedge E]{\alpha}{
      -A \wedge B &
      \infer*{\alpha}{[-A]} &
      \infer*{\alpha}{[-B]}
      }
    \end{gathered}
    \\
    \begin{gathered}
      \infer[+\vee I1]{+A \vee B}{
        +A
      }
    \end{gathered}
     & &
    \begin{gathered}
     \infer[+\vee I2]{+A \vee B}{
       +B
     }
    \end{gathered}
     & &
    \begin{gathered}
      \infer[+\vee E]{\alpha}{
      +A \vee B &
      \infer*{\alpha}{[+A]} &
      \infer*{\alpha}{[+B]}
      }
    \end{gathered}
    \\
    \begin{gathered}
      \infer[-\vee I]{-A \vee B}{
      -A &- B
      }
    \end{gathered}
    & &
    \begin{gathered}
      \infer[-\vee E1]{-A}{
        -A \vee B
      }
    \end{gathered}
     & &
    \begin{gathered}
      \infer[-\vee E2]{-B}{
        -A \vee B
      }
    \end{gathered}
    \\
    \begin{gathered}
      \infer[+\neg I]{+ \neg A}{
        -A
      }
    \end{gathered}
     & &
    \begin{gathered}
     \infer[+\neg E]{- A}{
       +\neg A
     }
    \end{gathered}
    \\
    \begin{gathered}
      \infer[-\neg I]{- \neg A}{
        +A
      }
    \end{gathered}
    & &
    \begin{gathered}
      \infer[-\neg E]{+ A}{
        -\neg A
      }
    \end{gathered}
    \\
    \begin{gathered}
      \infer[+\to I]{+A \to B}{
        \infer*{+B}{[+A]}
      }
    \end{gathered}
     & &
   \begin{gathered}
     \infer[+ \to E]{+B}{
       +A \to B & +A
     }
   \end{gathered}
    \\
   \begin{gathered}
     \infer[- \to I]{- A \to B}{
       +A & -B
     }
   \end{gathered}
    & &
   \begin{gathered}
     \infer[- \to E1]{+A}{
       - A \to B
     }
   \end{gathered}
    & &
    \begin{gathered}
      \infer[- \to E1]{-B}{
        - A \to B
      }
    \end{gathered}
  \end{align*}
  }
\end{definition}

\begin{definition}[Coordination rules]
  \begin{align*}
    \begin{gathered}
      \infer[\bot]{\bot}{
        +A & -A
      }
    \end{gathered}
    & &
    \begin{gathered}
      \infer[\RAA]{\alpha^*}{
        \infer*{\bot}{[\alpha]}
      }
    \end{gathered}
  \end{align*}
\end{definition}
In the rules $+ \to I$ and RAA, the assumption $\alpha$ is discharged and no longer open.
If $\D$ has no open assumption, we say that $\D$ is closed.
Otherwise, we say that $\D$ is open.

If $\alpha$ can be derived from open assumptions $\Gamma \equiv \beta_1, \ldots, \beta_n$ by a derivation $\D$, we write $\D : \Gamma \vdash \alpha$.
If $\D : \Gamma \vdash \alpha$ for some $\D$, we write $\Gamma \vdash \alpha$.

We define \emph{reduction rules} on $\D$.
We call the set of these reduction $R$.

{\scriptsize
\begin{align*}
  \begin{gathered}
    \infer{\bot}{
      \infer{\alpha^*}{
        \infer*{\bot}{[\alpha]}
      } &
      \infer*{\alpha}{}
    }
  \end{gathered} & \ \implies \
  \begin{gathered}
    \infer*{\bot}{
      \infer*{\alpha}{}
    }
  \end{gathered} &
  \begin{gathered}
    \infer{+A}{
      \infer{- A \to B}{
        +A & - B
      }
    }
  \end{gathered} & \ \implies \
  \begin{gathered}
      +A
  \end{gathered} &
  \begin{gathered}
      \infer{-B}{
        \infer{- A \to B}{
          +A & - B
        }
      }
  \end{gathered} & \ \implies \
  \begin{gathered}
      -B
  \end{gathered}
\end{align*}
\begin{align*}
    \begin{gathered}
      \infer{-A}{
        \infer{- A \vee B}{
          -A & - B
        }
      }
    \end{gathered} & \ \implies \
    \begin{gathered}
        -A
    \end{gathered}
    &
    \begin{gathered}
      \infer{\alpha}{
        \infer{-A \land B}{-A} &
        \infer*{\alpha}{[-A]} &
        \infer*{\alpha}{[-B]}
      }        
    \end{gathered} & \ \implies
    \begin{gathered}
      \infer*{\alpha}{[-A]}
    \end{gathered}
    &
    \begin{gathered}
      \infer{+A}{
        \infer{-\neg A}{+A}
      }
    \end{gathered} & \ \implies \
    \begin{gathered}
      +A
    \end{gathered}
    \end{align*}
  \begin{align*}
    \begin{gathered}
      \infer{\bot}{
        \infer{+A \wedge B}{
          +A & +B
        } &
        \infer{-A \wedge B}{
          -A
        }
      }
    \end{gathered} & \ \implies \
    \begin{gathered}
      \infer{\bot}{
        +A & -A
      }
    \end{gathered} &
    \begin{gathered}
      \infer{\bot}{
        \infer{+A \to B}{
          \infer*{+B}{[+A]}
        } &
        \infer{-A \to B}{
          \infer*{+A}{} &
          -B
        }
      }
    \end{gathered} & \ \implies \
    \begin{gathered}
      \infer{\bot}{
        \infer*{+B}{\infer*{+A}{}} &
        -B
      }
    \end{gathered}   
  \end{align*}
  \begin{align*}
    \begin{gathered}
      \infer{+A}{
        \infer{+A \wedge B}{
          \infer*{\bot}{
            [-A \wedge B]
          }
        }
      }
    \end{gathered} & \ \implies \
    \begin{gathered}
      \infer{+A}{
        \infer*{\bot}{
          \infer{-A \wedge B}{[-A]}
        }
      }
    \end{gathered} &
    \begin{gathered}
      \infer{-A}{
        \infer{-A \vee B}{
          \infer*{\bot}{
            [+A \vee B]
          }
        }
      }
    \end{gathered} & \ \implies \
    \begin{gathered}
      \infer{-A}{
        \infer*{\bot}{
          \infer{+A \vee B}{[+A]}
        }
      }
    \end{gathered} &
    \begin{gathered}
      \infer{+A}{
        \infer{-\neg A}{
          \infer*{\bot}{[+\neg A]}
        }
      }
    \end{gathered} & \ \implies \
    \begin{gathered}
      \infer{+A}{
        \infer*{\bot}{
          \infer{+\neg A}{
            [-A]
          }
        }
      }
    \end{gathered}
  \end{align*}
  \begin{align*}
    \begin{gathered}
      \infer{\alpha}{
        \infer{-A \wedge B}{
          \infer*{\bot}{[+A \wedge B]}} &
        \infer*{\alpha}{[-A]} &
        \infer*{\alpha}{[-B]}
        }
    \end{gathered} & \Rightarrow
    \begin{gathered}
      \infer{\alpha}{
      \infer*{\bot}{
        \infer{+A \wedge B}{
          \infer{+A}{
            \infer{\bot}{
              \infer*{\alpha}{[-A]}
              [\alpha^*]
            }
          } &
          \infer{+B}{
          \infer{\bot}{
            \infer*{\alpha}{[-B]}
            [\alpha^*]
          }
          }
        }
      }
      }
    \end{gathered} &
    \begin{gathered}
      \infer{\alpha}{
        \infer{+A \vee B}{
          \infer*{\bot}{[-A \vee B]}} &
        \infer*{\alpha}{[+A]} &
        \infer*{\alpha}{[+B]}
        }
    \end{gathered} &  \Rightarrow
    \begin{gathered}
      \infer{\alpha}{
      \infer*{\bot}{
        \infer{-A \vee B}{
          \infer{-A}{
            \infer{\bot}{
              \infer*{\alpha}{[+A]}
              [\alpha^*]
            }
          } &
          \infer{-B}{
          \infer{\bot}{
            \infer*{\alpha}{[+B]}
            [\alpha^*]
          }
          }
        }
      }
      }
    \end{gathered}
  \end{align*}
  \begin{align*}
    \begin{gathered}
      \infer{+B}{
        \infer{+ A \to B}{
          \infer*{\bot}{[-A \to B]}
        } &
      \infer*{+A}{}
      }
    \end{gathered}
    & \ \implies \
    \begin{gathered}
      \infer{+B}{
        \infer*{\bot}{
          \infer{-A \to B}{
            \infer*{+A}{} &
            [-B]
          }
        }
      }
    \end{gathered} &
    \begin{gathered}
      \infer{+A}{
        \infer{-A \to B}{
          \infer*{\bot}{[+A \to B]}
        }
      }
    \end{gathered}& \ \implies \
    \begin{gathered}
      \infer{+A}{
        \infer*{\bot}{
          \infer{+A \to B}{
            \infer{+B}{
              \infer{\bot}{
                [+A] & [-A]
              }
            }
          }
        }
      }
    \end{gathered} &
    \begin{gathered}
      \infer{-B}{
        \infer{-A \to B}{
          \infer*{\bot}{[+A \to B]}
        }
      }
    \end{gathered}& \ \implies \
    \begin{gathered}
      \infer{-B}{
        \infer*{\bot}{
          \infer{+A \to B}{
            [+B]
          }
        }
      }
    \end{gathered}
  \end{align*}
}  

\section{Validity}\label{sec:validity}

\begin{definition}[Basic system]
  A basic system is a given construction of atomic formulas and its refutations, as $\mathbf{S}$ in~\cite{Prawitz1971}.
  If all statements in an inference rule are atomic, the rule is called \emph{atomic}.
  If a rule does not have a premise, the rule is called an \emph{axiom}.
  A basic system $\B$ is a set of the atomic rules in the form
  \begin{equation*}
    \infer{+a}{
      +a_1 & +a_2 & \ldots & +a_n
    }
  \end{equation*}
  including the case $n=0$ and axioms $-b_1, \ldots, b_m$.
  We keep rules with premises only for assertions $+p$, because we want to keep $\B$ essentially intuitionistic.
  By abusing notation, $\B$ also denotes the minimal set of judgements closed under $\B$ and $\C$.

  Let $a \lesssim b$ if $+b$ appears as an assumption and $+a$ as a conclusion in $\B$.
  To avoid circular reasoning, we assume that there is no infinitely descendant sequence $\ldots \lesssim a_n \lesssim a_{n-1} \lesssim \ldots \lesssim a_1$.
  \emph{height} $h(a)$ for $a \in \B$ is the ordinal defined by
  \begin{itemize}
    \item $h(a) = 0$ if $\vdash_{\B} +a$.
    \item $h(b) = 0$ if $\vdash_{\B} -b$.
    \item $h(a) = \inf(\{ h(b) \mid +a_1, \ldots, +a_n, +b \vdash_{\B} +a \}) + 1$.
  \end{itemize}
  Because the relation $\lesssim$ is well-founded, $h(a)$ is defined for all $a \in \B$.
\end{definition}

We proceed the definition of validity.
We define closed valid derivations of statements $\alpha$ by induction on $\alpha$.

As Prawitz definition~\cite{Prawitz1971}, an open derivation $\D : \beta_1, \ldots,\beta_n \vdash \alpha$ is valid if and only if, when any closed valid derivation $\D_i: \vdash \beta_i$ is substituted to each $\beta_i$, the derivation obtained $\D[\D_1/\beta_1, \ldots, \D_n/\beta_n]: \vdash \alpha$ is valid.
Here, $\D[\D_1/\beta_1, \ldots, \D_n/\beta_n]$ means the derivation obtained by substituting  $\D_i: \vdash \beta_i$ to every $\beta_i$ in $\D$.

\subsection{$\bot$}
A derivation $\D: \vdash \bot$ is valid if it is normalizable to a derivation $\D_1$, in which all rules are either in $\B$ or $\C$.

\subsection{Atomic statements}

A closed derivation $\D: \vdash \alpha$ for an atomic statement $\alpha$ is valid if it can be reduced to either
\begin{enumerate}
  \item $\D$ is normalizable to a derivation in the form \[
    \infer{\alpha}{
      \infer*[\D_1]{\beta_1}{} & \ldots &
      \infer*[\D_n]{\beta_n}{}
    }
  \]
  where all $\D_1, \ldots, \D_n$ are valid.
  The definition above is well-defined by transfinite-induction on $h(\alpha)$.
  \item $\D$ can be reduced to a form \[
    \infer{\alpha}{
      \infer*[\D_1]{\bot}{
        [\alpha^*]
      }
    }
  \]
  and $\D_1$ is valid.
\end{enumerate}
The definition looks circular, but we can resolve circularity by Knaster-Tarkski's fixed point theorem~\cite{Tarski1955}.
Let $S$ be a set of valid derivation $\D: \vdash \alpha$.
Let $S^*$ be the set of normalizable derivations in the form
\[\infer{\alpha^*}{
  \infer*[\D_1]{\bot}{
    [\alpha]
  }
}
\]
in which $\D_1[\D/\alpha]: \vdash \bot$ can be reduced to a valid derivation for any valid derivation $\D$ of $\alpha$.
If $S \subseteq S'$, $S^* \supseteq S'^*$.

Consider $F_\alpha(S) = R_{\alpha} \cup S^*$ where $R_{\alpha}$ is the derivation normalizable to a derivation of $\alpha$ only using rules in $\B$ and $S$ is the set of valid derivation of $\alpha^*$.
If $S$ is the set of valid derivation of $-p$, $F_{+p}(S)$ is the set of valid derivation of $+p$.
Similarly, If $S$ is the set of valid derivation of $+p$, $F_{-p}(S)$ is the set of valid derivation of $-p$.

Because the operator $F_{+p}(F_{-p}(S))$ on $S$ is monotone, by Knaster-Tarkski's fixed point theorem, it has a unique least fixed point $X = F_{+p}(F_{-p}(X))$.
We define the element of $X$ as a valid derivation of $+p$ and $F_{-p}(X)$ as a valid derivation of $-p$.
Note that the choosing $X$ as the least fixed point, not largest one, is arbitrary choice.

\subsection{Logical connective}

A closed derivation $\D^+: \vdash +A \land B$ for a statement $+A \land B$ is valid if it can be reduced to either
\begin{enumerate}
  \item $\D^+$ is normalizable to a derivation in the form \[
    \infer{+A \land B}{
      \infer*[\D^+_1]{+A}{} &
      \infer*[\D^+_2]{+B}{}
    }
  \] in which $\D^+_1$ and $\D^+_2$ are valid.
  \item $\D^+$ is normalizable a form \[
    \infer{+A\land B}{
      \infer*[\D^+_*]{\bot}{
        [-A \land B]
      }
    }
  \]
  and $\D^+_*$ is valid.
\end{enumerate}
and a closed derivation $\D^-: \vdash -A \land B$ for a statement $-A \land B$ is valid if it can be reduced to either
\begin{enumerate}
  \item $\D^-$ is normalizable to a derivation in the form \[
    \infer{-A \land B}{
      \infer*[\D^-_1]{-A}{}
    }
  \] in which $\D^-_1$ is valid.
  \item $\D^-$ is normalizable to a derivation in the form \[
    \infer{-A \land B}{
      \infer*[\D^-_2]{-B}{}
    }
  \] in $\D^-_2$ is valid.
  \item $\D$ is normalizable a form \[
    \infer{-A\land B}{
      \infer*[\D^-_*]{\bot}{
        [+A \land B]
      }
    }
  \]
  and $\D^-_*$ is valid.
\end{enumerate}
Because validity of $\D^+_*$ depends on the notion of valid derivations of $- A \land B$, and validity of $\D^-_*$ depends on the notion of valid derivations of $+A \land B$, the definition above is circular.
We brake the circular by the fixed point construction as in the case of atomic formulas.

The notion of validity for $+A \lor B$ and $-A \lor B$ is symmetric to the case of $A \land B$.

The notion for $+A \to B$ and $- A \to B$ is defined as follows.
A closed derivation $\D^+: \vdash +A \to B$ for a statement $+A \to B$ is valid if it can be reduced to either
\begin{enumerate}
  \item $\D^+$ is normalizable to a derivation in the form \[
    \infer{+A \to B}{
      \infer*[\D^+_{\to}]{+A}{+B}
    }
  \] in which $\D^+_{\to}$ is valid.
  \item $\D^+$ is normalizable a form \[
    \infer{+A\to B}{
      \infer*[\D^+_*]{\bot}{
        [-A \to B]
      }
    }
  \]
  and $\D^-_*$ is valid.
\end{enumerate}
and a closed derivation $\D^-: \vdash -A \to B$ for a statement $-A \to B$ is valid if it can be reduced to either
\begin{enumerate}
  \item $\D^-$ is normalizable to a derivation in the form \[
    \infer{-A \to B}{
      \infer*[\D^+_1]{+A}{}&
      \infer*[\D^-_2]{-B}{}&
    }
  \] in which $\D^+_1$ and $\D^-_2$ are valid.
 \item $\D$ is normalizable a form \[
    \infer{-A\to B}{
      \infer*[\D^-_*]{\bot}{
        [+A \to B]
      }
    }
  \]
  and $\D^-_1$ is valid.
\end{enumerate}
Again, we resolve the circularity using fixed point construction.

\section{All derivations are valid}\label{sec:justification}

This section proves all derivation $\D$ of BCL can be reduced to valid derivations.
We prove more general theorem.
\begin{theorem}
  For any rule $R$ in $\BL$, if $\D_1, \ldots, \D_n$ are valid,
  \[
  \infer[R]{\alpha}{
    \infer*[\D_1]{\beta_1}{} & \ldots
    \infer*[\D_n]{\beta_n}{} 
    }
  \]
  is valid.
\end{theorem}

\subsection{RAA}

Let $X(\alpha)$ be the set of closed valid derivations of $\alpha$ and $X(\alpha^*)$ the set of closed valid derivations of $\alpha^*$.
By
 \[
  X^*(\alpha^*) \subseteq X(\alpha)
\]
If $\D$ has a form
\begin{align*}
  \infer{\alpha}{
    \infer*[\D_1]{\bot}{[\alpha^*]}
  }
\end{align*}
and $\D_1$ is valid, $\D \in X^*(\alpha^*) \subseteq X(\alpha)$.

\subsection{$\bot$-rule}

Assume that $\D$ ends with $\bot$-rule.
\[
  \infer{\bot}{
    \infer*[\D_1]{\alpha}{} &
    \infer*[\D_2]{\alpha^*}{}
  }
\]

By induction hypothesis, $\D_1$ and $\D_2$ are valid.
By the definition of validity, $\D_1$ is normalizable to a form in which the last rule is either $\B$-rule, introduction rule of a logical connective or $\RAA$.
The same holds for $\D_2$.

Assume either $\D_1$ or $\D_2$, say $\D_1$, is normalizable to a form ending $\RAA$.
Then $\D$ is normalizable to a form
\[
  \infer{\bot}{
    \infer{\alpha}{
      \infer*[\D^*_1]{\bot}{[\alpha^*]}
    } &
    \infer*[\D_2]{\alpha^*}{}
  }
\]
in which $\D^*_1$ is valid, because $\D$ is valid.
By reducing $\bot$-$\RAA$ pair, we get 
\[
  \infer*[\D^*_1]{\bot}{
    \infer*[\D_2]{\alpha^*}{}
  }
\]
Because $\D_1^*$ is valid, the derivation above is normalizable to a valid derivation.
By definition of validity for derivations of $\bot$, $\D$ is valid.

If neither $\D_1$ nor $\D_2$ is normalizable to a form ending $\RAA$, we prove the theorem by induction on $\alpha$.

\subsubsection{The $\alpha$ is atomic}

\[
  \infer{\bot}{
    \infer{+p}{
      \infer*[\D_1]{+q_1}{} &
      \ldots
      \infer*[\D_n]{+q_n}{}
    } &
    \infer{-p}{}
  }
\]
By the induction hypotheses, $\D$ is valid.

\subsubsection{Logical rules}
We show the case that $\alpha = +A \to B, \alpha^* = -A \to B$.
The proofs of other cases are similar.

\[
  \infer{\bot}{
    \infer{+A \to B}{
      \infer*[\D_1]{+B}{[+A]}
    } &
    \infer{-A \to B}{
      \infer*[\D_2]{+A}{} &
      \infer*[\D_3]{-B}{}
    }
  }
\]
We reduce it to
\[
  \infer{\bot}{
    \infer*[\D_1]{+B}{\infer*[\D_2]{+A}{}}
    \infer*[\D_3]{-B}{}
  }
\]
Because $\D_1, \D_2$ are valid, $\D_1[\D_2/+A]$ is also valid.
By induction hypothesis on $B$, the theorem holds.

\subsection{Atomic rules and logical rules}
If $R$ is an atomic rule or an introduction rule, the theorem is trivial from the definition of validity.
Therefore we consider the case in which $R$ is an elimination rule.
In particular, we consider $E\lor$-rule.
Consider the case that $\D$ is normalized to the following form.
\[
  \infer{\alpha}{
    \infer{+A \vee B}{
      \infer*[\D_1]{\bot}{[-A \vee B]}}
       &
      \infer*[\D_2]{\alpha}{[+A]} &
      \infer*[\D_3]{\alpha}{[+B]}
    }
\]
Because $\D$ is valid, we can assume that $\D_1, \D_2, \D_3$ are valid.
Because $\D_2, \D_3$ are valid and an introduction rule preserves validity,
\[
  \infer{-A \vee B}{
    \infer{-A}{
      \infer{\bot}{
        \infer*[\D_2]{\alpha}{[+A]}
        [\alpha^*]
      }
    } &
    \infer{-B}{
    \infer{\bot}{
      \infer*[\D_3]{\alpha}{[+B]}
      [\alpha^*]
    }
    }
  }
\]
is valid.
Combining the fact that $\D_1$ is valid, the theorem holds.

\section{Conclusions and philosophical remarks}

In conclusion, this paper has introduced a novel notion of validity in the style of Prawitz, specifically tailored to Rumfitt's bilateral classical logic (BCL). Through this framework, we have demonstrated that all rules, including the principle of contradiction and reductio ad absurdum (RAA), can be justified.

A significant challenge encountered in defining validity was the apparent circularity arising from RAA. This challenge was resolved through the application of fixed-point construction, facilitated by Knaster-Tarski's fixed point theorem.

It is worth noting that Knaster-Tarski's fixed point theorem relies on impredicative comprehension. Although there have been efforts to establish it within a constructive context~\cite{Curi2015-ji}, the general form remains unproven through constructive means. Consequently, our approach does not provide a constructive account of classical logic.

However, this does not diminish the significance of our work. Knaster-Tarski's theorem, while impredicative, does not hinge on the principle of excluded middle. This observation suggests that the principle of bivalence may not be an essential prerequisite for justifying classical logic. Moreover, our results underscore the potential for mathematical principles, rather than logical ones, to serve as a foundation for justifying logical principles.

% \section*{Acknowledgment}

% We would like to acknowledge the collaboration with Ukyo Suzuki~\cite{SuzukiManuscript-SUZOTN} as the foundation for this work. This collaboration played a significant role in shaping the paper. We extended the original work by allowing base systems to incorporate rules, not just axioms, thereby enhancing its scope. Additionally, we undertook substantial revisions and provided clarification for the philosophical aspects of the paper.

% It is important to note that while many aspects of this paper are indebted to Ukyo Suzuki's contributions, the decision was made for the authorship to be assumed by the present author. This decision was prompted by the unavailability of Suzuki Ukyo for further comments and input on the work.

\bibliographystyle{plain} % We choose the "plain" reference style
\bibliography{prawitz-style-validity-for-bcl.bib} % Entries are in the refs.bib file
\end{document}